\documentclass{amsart}
\usepackage{amsmath,dsfont,amsfonts,amssymb,amsxtra,latexsym,amscd,enumerate,amsthm}

\newcommand{\ls}{\lesssim}
\newcommand{\gs}{\gtrsim}

\newcommand{\R}{\mathbb{R}}
\newcommand{\C}{\mathbb{C}}
\newcommand{\Z}{\mathbb{Z}}
\newcommand{\N}{\mathbb{N}}

\DeclareFontFamily{U}{mathx}{\hyphenchar\font45}
\DeclareFontShape{U}{mathx}{m}{n}{ <5> <6> <7> <8> <9> <10> <10.95>
  <12> <14.4> <17.28> <20.74> <24.88> mathx10 }{}
\DeclareSymbolFont{mathx}{U}{mathx}{m}{n}
\DeclareMathAccent{\widecheck}{0}{mathx}{"71}

\newcommand{\eps}{\epsilon} 

\DeclareMathOperator{\Real}{Re}
\DeclareMathOperator{\Imag}{Im}

\newtheorem{thm}{Theorem} \newtheorem{cor}[thm]{Corollary}
\newtheorem{pro}[thm]{Proposition} \newtheorem{lem}[thm]{Lemma}

\theoremstyle{remark} \newtheorem{rmk}[thm]{Remark}
\theoremstyle{definition} 

\numberwithin{equation}{section} \numberwithin{thm}{section}

\begin{document}
\title[Critical results for fractional Hartree-type
equations]{Critical well-posedness and scattering results for
  fractional Hartree-type equations}

\author[S.~Herr]{Sebastian Herr} \address[S.~Herr]{Fakult\"at f\"ur
  Mathematik, Universit\"at Bielefeld, Postfach 10 01 31, 33501
  Bielefeld, Germany} \email{herr@math.uni-bielefeld.de}

\author[C.~Yang]{Changhun Yang} \address[C.~Yang]{Department of
  Mathematical Sciences, Seoul National University, Seoul 151-747,
  Republic of Korea} \email{maticionych@snu.ac.kr}

\begin{abstract}
  Scattering for the mass-critical fractional Schr\"odinger equation
  with a cubic Hartree-type nonlinearity for initial data in a small ball in the scale-invariant space of three-dimensional radial and square-integrable initial data is established. For this, we prove a  bilinear estimate for
  free solutions and extend it to perturbations of bounded quadratic variation. 
  This result is shown to be sharp by proving the unboundedness of a third order derivative of the flow map in  the super-critical range.
\end{abstract}
\subjclass[2010]{Primary: 35Q55; Secondary: 35Q40} \thanks{The authors
  acknowledge support by the German Science Foundation, IRTG 2235.}
\maketitle

\section{Introduction}
Let $n \in \N$, $1\leq \alpha\leq 2$, and $\sigma \in \R$. We consider
the following initial value problem for a fractional Schr\"odinger
equation with a cubic Hartree-type nonlinearity:
\begin{equation}\label{eqn}
  \begin{split}
    -i\partial_t u+(-\Delta)^{\frac{\alpha}{2}}u&=
    \sigma (|\cdot|^{-\alpha}*|u|^{2})u\\
    u(0,\cdot)&=\varphi
  \end{split}
\end{equation}
Here, the unknown is a function $u:(-T,T)\times \R^n\to \C$, the
initial datum is $\varphi:\R^n \to \C$ and
$(-\Delta)^{\frac{\alpha}{2}}$ is defined as the spatial Fourier
multiplier with symbol $|\cdot|^\alpha$ on $\R^n$, and $*$ denotes
spatial convolution.  We will consider initial data
$\varphi \in H^s(\mathbb{R}^{n})$ and solutions will be continuous
curves in $H^s(\mathbb{R}^{n})$.

We may rescale solutions according to
\begin{equation}\label{scale}
  u(t,x)\rightarrow u_{\lambda}(t,x):=\lambda^{\frac{n}{2}}u(\lambda^\alpha t,\lambda x),
\end{equation}
for fixed $\lambda>0$. The mass
\[M(u(t)):=\|u(t)\|_{L^2(\mathbb{R}^{n})}^2\]
of sufficiently smooth and decaying solutions $u$ of \eqref{eqn} is
conserved and invariant under this rescaling, i.e.\
$M(u_\lambda(t))=M(u(t))=M(\varphi)$ for any $t\in\mathbb{R}$. For
this reason the equation \eqref{eqn} is referred to as being mass-critical.

In addition, for sufficiently smooth and decaying solutions $u$ of \eqref{eqn}, the
energy
\[
E(u(t)):=\frac12\langle(-\Delta)^{\frac{\alpha}{2}}u(t),u(t)\rangle
+\frac{\sigma}{4}\langle (|\cdot|^{-\alpha}*|u(t)|^{2})u(t) ,
u(t)\rangle,\]
is conserved and the Sobolev space
$H^{\frac{\alpha}{2}}(\mathbb{R}^{n})$ serves as the energy space for
equation \eqref{eqn}. Here, $\langle \cdot ,\cdot \rangle$ is the
complex inner product in $L^2(\mathbb{R}^n)$.

The Cauchy problem with Hartree-type nonlinearities has been studied
intensively. If $\alpha=1$ and $n=3$, in which case \eqref{eqn} arises
as a model system for the dynamics of boson stars, Lenzmann and the
first named author proved local well-posedness for radial initial data
in the full subcritical range $s>0$ using $X^{s,b}$-spaces
\cite{HL-14}.
In the present paper the case of generalized dispersion,
i.e.\ $\alpha>1$, will be addressed.
In \cite{KLetal-13},
Kirkpatrick, Lenzmann and Staffilani rigorously derived a
fractional Schr\"odinger equation with a
cubic power-type nonlinearity if $n=1$ as the continuum limit of certain discrete physical
systems with long-range lattice interactions. As an
open problem they suggested that their argument might be generalized to the
fractional Schr\"odinger equation with other nonlinear terms in higher
dimension. Equation \ref{eqn} with $\alpha=2$ with
Hartree-nonlinearity, i.e.\ convolution with $|\cdot|^{-1}$, in $n=3$ was derived from the quantum theory of
large systems of bosons.  With regard to well-posedness of
\eqref{eqn}, Cho, Hajaiej, Hwang, and Ozawa showed global
well-posedness in the critical space for sufficiently small radial
initial data if $\frac{2n}{2n-1}<\alpha<2$ by using radial Strichartz
estimates \cite{CHetal-13}. Our main result fills the gap in the range
$1<\alpha\le\frac{2n}{2n-1}$.

Concerning the the scattering problem associated with \eqref{eqn} in the
case $\alpha=1$, the first named author and Tesfahun \cite{HT-15}
proved scattering of solutions for small radial initial data with
$s>0$ in the case of Yukawa potentials in $n=3$, while in case of the
Coulomb potential a modified scattering result has been established by
Pusateri \cite{P-14}. The classical
case $\alpha=2$ has been treated in \cite{HT-87}. In the fractional case $1<\alpha<2$, Cho, Hwang
and Ozawa \cite{CHetal-16} proved scattering for small initial data
when $s>\frac{2-\alpha}{2}$ in the case of generalized potentials
including the Yukawa potentials in $n\ge3$. 

For a more complete account on previous and related results, we refer to \cite{CHetal-13,CHetal-16}. Concerning more references to the physics literature, we refer to \cite{KLetal-13}.

We address the question of well-posedness and scattering of
\eqref{eqn} in the critical space in the range $1< \alpha\leq 2$.  To
obtain a result in the full range, we apply a contraction argument in a function space whose construction is based on the space of bounded quadratic variation $V^2$ and we extend a bilinear estimate for free
solutions to this space. In part, the strategy of proof is similar to \cite{HT-15}, but here we work in the critical regime.

In the super-critical range, i.e.\ $s<0$, we provide a counterexample
which implies the discontinuity of the flow map.

Our aim is to prove the existence and scattering of solutions to the
IVP \eqref{eqn} in the critical space $L^2(\R^3)$, i.e. we will focus
on $n=3$. We will consider the subspace of radial functions.  For
$s\in \R$ define
\[H^s_{rad}(\R^3):=\{ \varphi \in H^s(\R^3): \exists \,
\varphi_0:[0,\infty)\to \R \text{ s.th. } \varphi(x)=\varphi_0(|x|)
\text{ a.e.}\},\]
with norm $\|\cdot\|_{H^s}$.  We write
$L^2_{rad}(\R^3):=H^0_{rad}(\R^3)$.  Define
$(\mathbf{S}\varphi)(t,x)=(S(t)\varphi)(x)$, for
$\widehat{S(t)\varphi}(\xi)=e^{-it
  |\xi|^\alpha}\widehat{\varphi}(\xi)$.
Let us state our first main result on global well-posedness and
scattering for radial initial data which is small in the critical
space.
\begin{thm}\label{thm:main} Let $1<\alpha\leq 2$.
  There exists $\delta>0$, such that for all
  $\varphi\in L_{rad}^{2}(\R^3)$ satisfying
  $\|\varphi\|_{L^2}\leq \delta$, there exists a global solution
  $u\in C_b(\R,L_{rad}^{2}(\R^3))$ of \eqref{eqn}. $u$ is unique in a
  certain subspace and the flow map $\varphi \to u$ is smooth.

  Moreover, the solution scatters as $t\rightarrow\pm \infty$, i.e.\
  there exist $\varphi_\pm \in L_{rad}^{2}(\R^3)$, such that
  \[
  \|u(t)-S(t)\varphi_\pm\|_{L^2(\R^3)}\to 0 \quad (t\to \pm \infty).
  \]
\end{thm}

In order to state the second result, let $P_{>\Lambda}$ be the projector onto frequencies greater than
$\Lambda$, see Subsection \ref{subsec:not}, and let us define
\begin{equation}\label{def:BdeltaR}
    B_{r,\Lambda} :=\big\{\varphi\in L_{rad}^2(\R^3) : \| \varphi\|_{L^2} \leq r, \|P_{>\Lambda} \varphi
    \|_{L^2} \leq \eta r^{-1} \big\},
\end{equation}
for $r,\Lambda\geq 1$ and some parameter $0<\eta \ll 1$, which will be
fixed in Subsection \ref{subsec:proof-local} independently of $r,\Lambda$.
Notice that for each $\varphi\in L_{rad}^2(\R^3)$
and $\eta>0$ there exist $\Lambda>0$ such that
$\|P_{> \Lambda}\varphi\|_{L^2}\leq \eta r^{-1}$. Therefore, for any $\varphi\in L_{rad}^2(\R^3)$ and any $r\geq \|\varphi\|_{L^2}$ there exists $\Lambda\geq 1$, such that $\varphi\in B_{r,\Lambda}$.  For large radial initial
data in the critical space, we have local well-posedness.

\begin{thm}\label{thm:local} Let $1<\alpha\leq 2$.
  For all $r,\Lambda\geq 1$ and all
  $\varphi\in B_{r,\Lambda}\subset L^2_{rad}(\R^3)$, there exists
  $T=T(r,\Lambda)$ and a
  solution $u\in C([0,T],L_{rad}^{2}(\R^3))$ of \eqref{eqn}. $u$ is
  unique in a certain subspace and the flow map $\varphi \to u$ is
  smooth.
\end{thm}

Our next result shows that Theorem~\ref{thm:local} is optimal.
\begin{thm}\label{thm:illposed}
  Let $1<\alpha\leq 2$, $s<0$. If the flow map $\varphi\mapsto u$ exists (in a small neighborhood of the origin) as a map from $H^s(\R^3)$ to $C([0,T],H^s_{rad}(\R^3))$, it fails to be $C^3$ at the origin.
\end{thm}
The proof is based on an adaptation of the counterexample for the case
$\alpha=1$ from
\cite{HL-14}.

\subsection{Notation}\label{subsec:not}
Dyadic numbers $\lambda \in 2^\Z$ will always be denoted by greek
letters, e.g.\ $\mu,\lambda,\lambda_1,\lambda_2$, and sums with
respect to greek letters are implicitly assumed to range over (subsets
of) $2^\Z$.

Let $\rho\in C_c^\infty(-2,2)$ be even and satisfy $\rho(s)=1$ for
$|s|\leq 1$. For $\chi(\xi):=\rho(|\xi|)-\rho(2|\xi|)$ define
$\chi_\lambda(\xi)=\chi(\lambda^{-1}\xi)$. Then,
$\sum_{\lambda \in 2^\Z}\chi_\lambda =1$ on $\R^n\setminus\{0\}$ at it
is locally finite.  We define the (spatial) Fourier localization
operator $P_\lambda f=\mathcal{F}^{-1}(\chi_\lambda \mathcal{F}f)$.
Further, we define
$\chi_{\leq \lambda}=\sum_{\mu \in 2^\Z: \mu \leq \lambda }\chi_\mu$
and
$P_{\leq \lambda} f=\mathcal{F}^{-1}(\chi_{\leq \lambda}
\mathcal{F}f)$,
$P_{>\lambda} f=f-P_{\leq \lambda}f $.  Let
$\widetilde{\chi}_{\lambda}=\chi_{\lambda/2}+\chi_{\lambda}+\chi_{2\lambda}$
and
$\widetilde{P}_\lambda f=\mathcal{F}^{-1}(\widetilde{\chi}_\lambda
\mathcal{F}f)$.  Then $\widetilde{P}_\lambda P_\lambda=P_\lambda\widetilde{P}_\lambda=P_\lambda$.

\section{Bilinear estimates for radial functions}\label{sec:bil}

\subsection{Free solutions}\label{subsec:free}
Since the characteristic hypersurface in $\R^{1+n}$ defined by the
phase function $|\xi|^\alpha$ has $n$ nonvanishing principal
curvatures in the case $\alpha>1$, there are similar Strichartz
estimates as for the Schr\"odinger equation, up to a loss of
derivatives dictated by scaling. For the following Lemma, its proof and more information we refer to \cite{COetal-11}.
\begin{lem}\label{lem:str}
  Let $1<\alpha\leq 2$, $q> 2$, $r\geq 2$,
  $\frac{2}{q}+\frac{n}{r}=\frac{n}{2}$,
  $\theta=\frac{n}{2} (2-\alpha)(\frac12-\frac{1}{r})$. Then,
  \[
  \|\mathbf{S}\varphi\|_{L^q_t(\R, L^r_x(\R^n))}\ls
  \|\varphi\|_{H^\theta(\R^n)}.
  \]
\end{lem}
Next, we adapt
\cite[Lemma 3.2]{HT-15} to the case of generalized dispersion.
\begin{lem}\label{lm:eq:I}
Let $1\leq \alpha\leq 2$.
  Consider the integral
  \[
  I(\phi,\psi)(\tau,\xi) =\int\phi(|\eta|)\psi(|\xi-\eta|)
  \delta(\tau-|\eta|^{\alpha}+|\xi-\eta|^{\alpha}) d\eta
  \]
  for smooth $\phi$ and $\psi$ supported in $[-r,r]$ and $[-R,R]$, for
  some $r,R>0$.  Then for
  $0\le\tau\le \alpha\max\{ r,R \}^{\alpha-1} |\xi| $,
  \begin{equation}
    I(\phi,\psi)(\tau,\xi)
    =
    \frac{2\pi}{\alpha|\xi|}
    \int_{a(\tau,|\xi|)}^{\infty}
    \phi(\rho)\psi(\omega(\tau,\rho))
    \omega(\tau,\rho)
    ^{2-\alpha}
    \rho d\rho  
  \end{equation}
  where
  \begin{equation}\label{def:aomega}
    a(\tau,|\xi|)=\frac{|\xi|^2+\tau^{2/ \alpha}}{2|\xi|} \qquad \text{ and }\qquad
    \omega(\tau,\rho)
    =(\rho^{\alpha}-\tau)^{1/\alpha}.
  \end{equation}
  Furthermore,
  \begin{equation}\label{vanish}
    I(\phi,\psi)(\tau,\xi)=0,	\text{ if } ~ \tau > \alpha\max\{ r,R \}^{\alpha-1} |\xi|.
  \end{equation} 
\end{lem}
\begin{proof} As in \cite[pp.\ 8--9]{HT-15}, the proof is an
  straight-forward modification of the argument for $\alpha=1$ from
  \cite[Lemma~4.4]{FK-00}.  First, we check that the delta function
  and support condition on $\phi,\psi$ restrict the range of
  $\tau$. By the mean-value theorem we have
  \begin{align*}
    |\tau|
    =\big||\eta|^{\alpha}
    -|\xi-\eta|^{\alpha}\big| 
    \le \alpha
    \max\{ |\eta|,|\xi-\eta| \}^{\alpha-1} |\xi|,
  \end{align*}
  implying the second claim \eqref{vanish}.

  Now, let $0\le\tau\le \alpha\max\{ r,R \}^{\alpha-1} |\xi| $. Using
  $\delta-$calculus, we can write
  \begin{align}
    \begin{aligned}\label{eq:detlam}
      &\delta(\tau-|\eta|^{\alpha}+|\xi-\eta|^{\alpha})\\
      =&\big|\tau-|\eta|^{\alpha}-|\xi-\eta|^{\alpha}\big|
      \delta\big(|\xi-\eta|^{2\alpha} -(\tau-|\eta|^{\alpha})^2 \big) \\
      =&2\big(|\eta|^{\alpha}-\tau\big)
      \delta\big((|\xi|^2-2\xi\cdot\eta+|\eta|^2)^\alpha-(\tau-|\eta|^{\alpha})^2
      \big),
    \end{aligned}
  \end{align} 
  having used $0\le|\xi-\eta|^{\alpha}=|\eta|^{\alpha}-\tau $ within
  the support of $\delta$.  Introduce polar coordinates for
  $\eta=\rho\theta$, where $\rho=|\eta|$ and
  $ \theta=\frac{\eta}{|\eta|}\in S^{2}.$ Then
  $ d\eta=\rho^{2}dS_\theta d\rho$. Further, for the new variable
  $b=\theta\cdot\frac{\xi}{|\xi|}$ we obtain
  \[
  dS_\theta=dS_{\theta'}db,\ \text{for} \ \theta'\in S^1 \text{ and }
  d\eta=\rho^2 d\rho dS_{\theta'}db.
  \]
  Because of
  $ |\xi-\eta|^\alpha = |\eta|^\alpha-\tau = \rho^{\alpha}-\tau $,
  we have $|\xi-\eta|=\omega(\tau,\rho)$.  Now, with
  \[
  g^{\tau,\xi}_{\rho}(b) =(|\xi|^2-2|\xi|\rho
  b+\rho^2)^\alpha-(|\eta|^{\alpha}-\tau)^2 ,\]
  the integrand is independent of $\theta'$ and we obtain
  \[
  I(\phi,\psi)(\tau,\xi) =4\pi \int_{\tau^{\frac{1}{\alpha}}}^\infty
  \int_{-1}^{1} \phi(\rho)\psi(\omega(\tau,\rho)) (\rho^{\alpha}-\tau)
  \rho^2 \delta\big( g_{\rho}^{\tau,\xi}(b) \big) db d\rho.
  \]
  We use the delta function to set the value of $b$ to
  \begin{equation}\label{zeroofdelta}
    b_{\rho}^{\tau,\xi}:=\frac{|\xi|^2+\rho^2-(\rho^{\alpha}-\tau)^{2/\alpha}}{2|\xi|\rho},
  \end{equation}
  with the condition $b\le1$ that forces
  \begin{align*}
    |\xi|^2+\rho^2\le 2|\xi|\rho+(\rho^{\alpha}-\tau)^{2/\alpha}
  \end{align*}
  and since
  $( \rho^{\alpha}-\tau )^{2/\alpha} \le \rho^{2}-\tau^{2/\alpha}$,
  the domain of the $\rho$-integration is further restricted to
  \begin{equation}\label{intrange}
    \bigg\{\rho\ge\frac{|\xi|^2+\tau^{2/\alpha}}{2|\xi|} \bigg\}.
  \end{equation}
  We compute the integral over $b$ as
  \begin{equation}\label{intdelta}
    \int_{-1}^1 \delta\big(
    g^{\tau,\xi}_{\rho}(b)
    \big)
    db
    =\big|\frac{\mathrm{d}}{\mathrm{d}b}g^{\tau,\xi}_{\rho} (b_{\rho}^{\tau,\xi})\big|^{-1}=(2\alpha|\xi|\rho(\rho^\alpha-\tau)^{\frac{2(\alpha-1)}{\alpha}})^{-1},
  \end{equation}
  where $b_{\rho}^{\tau,\xi}$ is the value in \eqref{zeroofdelta}.
  With \eqref{intrange} and \eqref{intdelta}, we finally obtain
  \begin{align*}
    I(\phi,\psi)(\tau,\xi)=4\pi
    \int_{a(\tau,|\xi|)}^{\infty}
    \phi(\rho)\psi(\omega(\tau,\rho))
    \frac{\rho(\rho^{\alpha}-\tau)}{2\alpha |\xi|(\rho^\alpha-\tau)^{\frac{2(\alpha-1)}{\alpha}}}
    d\rho,
  \end{align*}
  which reduces to the desired form.
\end{proof}

\begin{pro}\label{pro:bil-v2}
  Let $1\leq \alpha\leq 2$.  Consider $u^{+}(t)=S(t)f$ and
  $v^{-}(t)=S(-t)g$, where f and g are radial.  Then for any $\mu>0$
  and $\lambda_1\ge\lambda_2>0$, we have
  \begin{equation}\label{ineq:uv}
    \|P_\mu(u_{\lambda_1}^{+}v_{\lambda_2}^{-})\|
    _{L_{t,x}^2(\mathbb{R}^{1+3})}
    \lesssim
    \mu
    \lambda_2^\frac{1-\alpha}{2}
    \|f_{\lambda_1}\|_{L_x^{2}(\mathbb{R}^{3})} 
    \|g_{\lambda_2}\|_{L_x^{2}(\mathbb{R}^{3})} 
  \end{equation}
\end{pro}

\begin{proof}
  The Fourier transform of a radial function is a radial function, and
  we may assume $\widehat{f},\widehat{g}\geq 0$.  We denote
  $\widehat{f}\chi_{\lambda_1}, \widehat{g}\chi_{\lambda_2}$ by
  $\psi_{\lambda_1}$, $\phi_{\lambda_2}$, respectively, and compute
  the space-time Fourier transform
  \begin{align*}
    &\mathcal{F}_{t,x} 
      \{P_{\mu}(u_{\lambda_1}^{+}v_{\lambda_2}^{-}) \}
      (\tau,\xi)   \\
    ={}&
         \int_{\mathbb{R}}e^{-it\tau}
         \chi_\mu(|\xi|)
         \int_{\mathbb{R}^{3}}
         \chi_{\lambda_1}(|\xi-\eta|)e^{-it|\xi-\eta|^\alpha}
         \hat{f}(\xi-\eta)
         \chi_{\lambda_2}(|\eta|)e^{it|\eta|^\alpha}
         \hat{g}(\eta)
         d\eta dt  \\
    ={}&\chi_\mu(|\xi|)
         \int_{\mathbb{R}^{3}}
         \psi_{\lambda_1}(|\xi-\eta|)
         \phi_{\lambda_2}(|\eta|)
         \delta(\tau-|\eta|^{\alpha}
         +|\xi-\eta|^{\alpha})
         d\eta.
  \end{align*}
  For $0\le\tau\le C_\alpha \lambda_1^{\alpha-1}\mu$,
  Lemma~\ref{lm:eq:I} implies
  \[
  \mathcal{F}_{t,x} \{P_{\mu}(u_{\lambda_1}^{+}v_{\lambda_2}^{-}) \}
  (\tau,\xi) \approx \chi_\mu(\xi)\frac{1}{|\xi|}
  \int_{a(\tau,|\xi|)}^{\infty}
  \phi_{\lambda_2}(\rho)\psi_{\lambda_1}(\omega(\tau,\rho))
  \omega(\tau,\rho) ^{2-\alpha} \rho d\rho,
  \]
  while for $\tau> C_\alpha \lambda_1^{\alpha-1}\mu$ there is no
  contribution.  If $-C_\alpha \lambda_1^{\alpha-1}\mu\le \tau\le 0$,
  we similarly obtain
  \begin{align*}
  &\mathcal{F}_{t,x} \{P_{\mu}(u_{\lambda_1}^{+}v_{\lambda_2}^{-}) \}
  (\tau,\xi) \\
\approx{}& \chi_\mu(\xi)\frac{1}{|\xi|}
  \int_{a(-\tau,|\xi|)}^{\infty}\phi_{\lambda_2}(\omega(-\tau,\rho))
  \psi_{\lambda_1}(\rho) \omega(-\tau,\rho) ^{2-\alpha} \rho d\rho,
  \end{align*}
  while for $\tau< -C_\alpha \lambda_1^{\alpha-1}\mu$ there is no
  contribution.  Hence,
  \[
  \|P_\mu(u_{\lambda_1}^{+}v_{\lambda_2^{-}})\|_{L_{t,x}^2(\mathbb{R}^{1+3})}^2
  := I_1 + I_2 ,\] where, by Plancherel's theorem,
  \begin{align*}
    I_1&\approx
         \int_{\mathbb{R}^3}
         \int\limits_{0}^{C_\alpha \lambda_1^{\alpha-1}\mu}
         \frac{\chi_\mu^2(|\xi|)}{|\xi|^2}
         \bigg|
         \int_{a(\tau,|\xi|)}^{\infty}
         [\rho\phi_{\lambda_2}(\rho)]
         [	\omega(\tau,\rho)
         ^{2-\alpha}
         \psi_{\lambda_1}(\omega(\tau,\rho))]
         d\rho 
         \bigg|^2
         d\tau d\xi ,\\
    I_2 &\approx  
          \int_{\mathbb{R}^3}
          \int\limits_{0}^{C_\alpha \lambda_1^{\alpha-1}\mu}
          \frac{\chi_\mu^2(|\xi|)}{|\xi|^2}
          \bigg|
          \int_{a(\tau,|\xi|)}^{\infty}
          [\rho\psi_{\lambda_1}(\rho)]
          [	\omega(\tau,\rho)
          ^{2-\alpha}
          \phi_{\lambda_2}(\omega(\tau,\rho))]
          d\rho 
          \bigg|^2
          d\tau d\xi.
  \end{align*}
  In polar coordinates
  $\xi\rightarrow (r,\theta) \in [0,\infty)\times S^2$, the first term
  is
  \begin{align*}
    I_1&\approx
         \int_{0}^\infty \chi_{\mu}^2(r)
         \int_{0}^{C_\alpha\lambda_1^{\alpha-1}\mu}
         \bigg|
         \int_{a(\tau,r)}^{\infty}
         [\rho\phi_{\lambda_2}(\rho)]
         [	\omega(\tau,\rho)
         ^{2-\alpha}
         \psi_{\lambda_1}(\omega(\tau,\rho))
         ]
         d\rho 
         \bigg|^2
         d\tau dr.
  \end{align*}
  The Cauchy-Schwarz inequality implies
  \begin{align*}
    &\bigg|
      \int_{a(\tau,|\xi|)}^{\infty}
      [\rho\phi_{\lambda_2}(\rho)]
      [	\omega(\tau,\rho)
      ^{2-\alpha}
      \psi_{\lambda_1}(\omega(\tau,\rho))
      ]
      d\rho 
      \bigg|^2\\
    \lesssim{}& \bigg(
                \int_{\mathbb{R}}
                |\phi_{\lambda_2}(\rho)\rho|^2 d\rho
                \bigg)\bigg(
                \int_{\mathbb{R}}
                |	\omega(\tau,\rho)^{2-\alpha}
                \psi_{\lambda_1}(\omega(\tau,\rho))
                \widetilde{\chi}_{\lambda_2}(\rho)
                |^2 d\rho
                \bigg)\\
    \lesssim{}&\|g_{\lambda_2}\|_{L^2(\R^3)}^2\bigg(
                \int_{\mathbb{R}}
                |	\sigma^{2-\alpha}
                \psi_{\lambda_1}(\sigma)|^2 \big(\frac{\lambda_1}{\lambda_2}\big)^{\alpha-1}
                d\sigma
                \bigg).
  \end{align*}
  Here, we used the change of variables $\sigma=\omega(\tau,\rho)$,
  hence $\sigma^{\alpha}=\rho^{\alpha}-\tau $ and
  $d\rho= (\frac{\sigma}{\rho})^{\alpha-1} d\sigma$, and in the domain
  of integration we have
  $|\frac{\sigma^{\alpha-1}}{\rho^{\alpha-1}} | \lesssim
  (\frac{\lambda_1}{\lambda_2})^{\alpha-1}$.  We obtain
  \begin{align*}
    I_1&
         \lesssim
         \mu
         \|g_{\lambda_2}\|_{L_x^{2}(\mathbb{R}^{3})}^2
         \lambda_1^{\alpha-1}\mu
         \big(\frac{1}{\lambda_1\lambda_2}\big)^{\alpha-1}
         \|f_{\lambda_1}\|_{L_x^{2}(\mathbb{R}^{3})}^2 \\
       &\lesssim
         \mu^2 \lambda_2^{1-\alpha}
         \|g_{\lambda_2}\|_{L_x^{2}(\mathbb{R}^{3})}^2
         \|f_{\lambda_1}\|_{L_x^{2}(\mathbb{R}^{3})}^2.
  \end{align*}
  Concerning $I_2$, we obtain
  \begin{align*}
    I_2&
         \lesssim
         \mu\|f_{\lambda_1}\|_{L_x^{2}(\mathbb{R}^{3})}^2
         \int_{0}^{C_\alpha\lambda_1^{\alpha-1}\mu}
         \int_{\mathbb{R}}
         |	\omega(\tau,\rho)^{2-\alpha}
         \phi_{\lambda_2}(\omega(\tau,\rho))
         |^2 d\rho d\tau
  \end{align*}
  along the same lines. Again,
  \[I_2 \lesssim \mu^2 \lambda_2^{1-\alpha}
  \|g_{\lambda_2}\|_{L_x^{2}(\mathbb{R}^{3})}^2
  \|f_{\lambda_1}\|_{L_x^{2}(\mathbb{R}^{3})}^2,\]
  by the same change of variables as above.
\end{proof}

\begin{cor}\label{cor:bil-v2} Let $1\leq \alpha\leq 2$.
  For all dyadic $\mu\ls \lambda_1\sim\lambda_2$ and spatially radial
  functions $u_j=S(\cdot)\varphi_j $, we have
  \begin{equation}\label{eq:bil-rad}
    \|P_{\leq \mu} (P_{\lambda_1} u_1 \overline{P_{\lambda_2} u_2})\|_{L^2(\R^{1+3})}\ls \mu^{\frac{3-\alpha}{2}}\Big(\frac{\mu}{\lambda_1}\Big)^{\frac{\alpha-1}{2}}\|P_{\lambda_1} \varphi_1 \|_{L^2(\R^3)}\|P_{\lambda_2} \varphi_2\|_{L^2(\R^3)}.
  \end{equation}
\end{cor}
\begin{proof}
  This follows by dyadic summation over $\mu'\leq \mu$ from
  Proposition \ref{pro:bil-v2}.
\end{proof}
\subsection{Transference}\label{subsec:trans}
Let $1\leq p<\infty$. We call a finite set $\{t_0,\ldots,t_K\}$ a partition if $-\infty<t_0<t_1<\ldots<t_K\leq \infty$, and denote the set of all partitions by $\mathcal{T}$. A corresponding step-function $a:\R\to L^2(\R^3)$ is called $U^p_{\mathbf{S}}$-atom if
\[
a(t)=\sum_{k=1}^K \mathbf{1}_{[t_{k-1},t_k)} (t)S(t)\varphi_k, \quad \sum_{k=1}^K\|\varphi_k\|_{L^2(\R^3)}^p=1,\quad \{t_0,\ldots,t_K\}\in \mathcal{T},
\]
and $U^p_{\mathbf{S}}$ is the atomic space. Further, let $V^p_{\mathbf{S}}$ be the space of all right-continuous $v:\R \to L^2(\R^3)$ satisfying
\[
\|v\|_{V^p_{\mathbf{S}}}:=\sup_{\{t_0,\ldots,t_K\}\in \mathcal{T}}\big(\sum_{k=1}^K\|S(-t_k)v(t_k)-S(-t_{k-1})v(t_{k-1})\|_{L^2(\R^3)}^p\big)^{\frac1p}.
\]
with the convention $S(-t_K)v(t_K)=0$ if $t_K=\infty$.
For the theory of $U^p_{\mathbf{S}}$ and $V^p_{\mathbf{S}}$, see e.g.\
\cite{HHetal-09,HHetal-10,KTetal-14}.
For $s\in \R$ let
\[
\|u\|_{X^s}=\Big(\sum_{\lambda \in 2^\Z} \lambda^{2s}\|P_\lambda
u\|_{V^2_\mathbf{S}}^2\Big)^{\frac12}.
\]
By the atomic structure of $U^2_{\mathbf{S}}$, estimates in $L^2$ for
free solutions transfer to $U^2_{\mathbf{S}}$-functions, hence to
$V^p_{\mathbf{S}}$ for $p<2$. However, transference to
$V^2_{\mathbf{S}}$ does not follow from the general theory of these
spaces. Nevertheless, we prove below that in case of the bilinear estimate of the previous
section it does hold true. This might also have applications in the case
$\alpha=1$, and the proof applies to certain other multilinear
estimates.

\begin{pro}\label{pro:bil-rad-v2} Let $1\leq \alpha\leq 2$.
  For all dyadic $\mu\ls \lambda_1\sim\lambda_2$ and spatially radial
  functions $u_1,u_2\in V^2_{\mathbf{S}}$, we have
  \begin{equation}\label{eq:bil-rad-v2}
    \|P_{\leq \mu} (P_{\lambda_1} u_1 \overline{P_{\lambda_2} u_2})\|_{L^2(\R^{1+3})}\ls \mu^{\frac{3-\alpha}{2}}\Big(\frac{\mu}{\lambda_1 }\Big)^{\frac{\alpha-1}{2}}\|P_{\lambda_1} u_1\|_{V^2_{\mathbf{S}}}\|P_{\lambda_2} u_2\|_{V^2_{\mathbf{S}}}.
  \end{equation}
\end{pro}
\begin{proof} {\it 1. Step:} Let
  $P=\sum_{\lambda \in F} P_\lambda$, with a finite set $F$ of
  dyadic numbers $\lambda$ of size $\lambda_1\sim\lambda_2$, such that
  $P P_{\lambda_j} =P_{\lambda_j}$ for $j=1,2$. We claim that
  \begin{equation}\label{eq:u4}
    \|P_{\leq \mu} |P w|^2\|_{L^2}\ls  \mu^{\frac{3-\alpha}{2}}\Big(\frac{\mu}{\lambda_1 }\Big)^{\frac{\alpha-1}{2}}\|w\|_{U^4_{\mathbf{S}}}^2
  \end{equation}
  for any $\mu\ls \lambda_1$ and spatially radial
  $w\in U^4_{\mathbf{S}}$. To prove \eqref{eq:u4}, let
  $V_\mu=(\widecheck{\chi}_{\leq 2 \mu})^2$. Then, $V_\mu \geq 0$ and
  we have the pointwise bound
  \[\chi_{\leq \mu}\ls \chi_{\leq 2\mu} \ast \chi_{\leq 2\mu} \ls
  \chi_{\leq 4\mu}\] on the Fourier side, which implies
  \[ \|P_{\leq \mu} |P w|^2 \|_{L^2}\ls \|V_\mu \ast |P
  w|^2\|_{L^2}\ls \|P_{\leq 4\mu} |P w|^2\|_{L^2}.\]
  The quantity
  \[n(f):=\|\Phi(f)\|_{L^4(\R^{3})}, \text{ for }\Phi(f)=\Big(V_\mu
  \ast |P f|^2\Big)^{\frac12},\]
  is subadditive, and
  $\|V_\mu \ast |P w|^2\|_{L^2}=\|n(w(t))\|_{L^4_t}^2$. Indeed,
  \begin{align*}
    \Phi^2(f_1+f_2)(x)={}&\int_{\R^3} V_\mu (x-y)|P f_1(y)+P
                           f_2(y)|^2 dy\\
    \leq{}&\int_{\R^3} V_\mu (x-y)|P f_1(y)+P
            f_2(y)||P f_1(y)| dy\\&{}+\int_{\R^3} V_\mu (x-y)|P f_1(y)+P
                                    f_2(y)||P f_2(y)| dy\\
    \leq{}& \Phi(f_1+f_2)(x)\Phi(f_1)(x)+\Phi(f_1+f_2)(x)\Phi(f_2)(x)
  \end{align*}
  by Cauchy-Schwarz, so $\Phi(f_1+f_2)\leq \Phi(f_1)+\Phi(f_2)$ and
  \[
  n(f_1+f_2)\leq \|\Phi(f_1)+\Phi(f_2)\|_{L^4}\leq n(f_1)+n(f_2)
  \]
  follows. Also, we obviously have $n(c f)=|c|n(f)$ for all $c\in\C$.
  Due to \eqref{eq:bil-rad} we have
  \begin{equation}\label{eq:n-est}
    \|n(S(t)\varphi)\|_{L^4_t}\leq C(\mu,\lambda_1)^{\frac12} \|\varphi\|_{L^2}
  \end{equation}
  for all radial $\varphi \in L^2(\R^3)$, where $C(\mu,\lambda_1)$
  denotes the constant in \eqref{eq:bil-rad}.  Let
  $w\in U^4_{\mathbf{S}}$ be radial with atomic decomposition
  \[
  w=\sum_{j} c_j a_j, \; \text{s.th. }\sum_{j} |c_j| \leq 2
  \|w\|_{U^4_{\mathbf{S}}}, \text{ and radial }
  U^4_{\mathbf{S}}\text{-atoms } a_j.
  \]
  We have
  \begin{equation}\label{eq:n-est-u4}
    \|n(w)\|_{L^4_t}\leq \sum_{j} |c_j| \|n(a_j)\|_{L^4_t} \ls
    C(\mu,\lambda_1)^{\frac12}\|w\|_{U^4_{\mathbf{S}}},
  \end{equation}
  provided that for any $U^4_{\mathbf{S}}$-atom $a$ the estimate
  \[
  \|n(a)\|_{L^4_t}\ls C(\mu,\lambda_1)^{\frac12}
  \]
  holds true. Indeed, let
  $a(t)=\sum_k\mathbf{1}_{I_k}(t) S(t)\varphi_k$, for some partition
  $(I_k)$ of $\R$ and radial $\varphi_k\in L^2(\R^3)$ satisfying
  $\sum_{k}\|\varphi_k\|_{L^2}^4\leq 1$. Then,
  \begin{align*}
    \|n(a)\|_{L^4_t}\leq{}& \Big\|\sum_k\mathbf{1}_{I_k}(t)
                            n(S(t)\varphi_k)\Big\|_{L^4_t}
                            \leq{} \Big(\sum_k  \|n(S(t)\varphi_k)\|_{L^4_t}^4\Big)^{\frac14}\\
    \ls{}& C(\mu,\lambda_1)^{\frac12}\Big(\sum_k
           \|\varphi_k\|_{L^2}^4\Big)^{\frac14}
           \ls{} C(\mu,\lambda_1)^{\frac12},
  \end{align*}
  where we used \eqref{eq:n-est} in the third inequality, which
  completes the proof of \eqref{eq:n-est-u4}. This implies
  \begin{equation}
    \|P_{\leq \mu} |P w|^2 \|_{L^2(\R^{1+3})}\ls \|n(w)\|_{L^4_t}^2\ls
    C(\mu,\lambda_1)\| w\|_{U^4_{\mathbf{S}}}^2,
  \end{equation}
  hence the claim \eqref{eq:u4}.

  {\it 2. Step:} Let $v_{j}:=P_{\lambda_j} u_j$, $j=1,2$. We may
  assume $\|v_j\|_{U^4_{\mathbf{S}}}=1$. The functions
  $w_\pm=v_1\pm v_2$ satisfy $w_\pm=P w_\pm$,
  $\|w_\pm\|_{U^4_{\mathbf{S}}}\ls 1$,
  \[
  \Real(v_1\overline{v_2})=\frac14 \Big(|w_+|^2-|w_-|^2\Big),
  \text{ and }\Imag(v_1\overline{v_2})=\Real(-i v_1\overline{v_2}).\]
The  estimate \eqref{eq:u4} yields
  \begin{align*}
    \|P_{\leq \mu} (v_1\overline{v_2}) \|_{L^2(\R^{1+3})}\ls{}& \|P_{\leq
                                                                \mu} |P w_+|^2 \|_{L^2(\R^{1+3})}+ \|P_{\leq \mu} |Pw_-|^2
                                                                \|_{L^2(\R^{1+3})}\\
    \ls{}& C(\mu,\lambda_1)\big(\|w_+\|_{U^4_{\mathbf{S}}}^2+\|w_-
           \|_{U^4_{\mathbf{S}}}^2\big)\\
    \ls{}& C(\mu,\lambda_1),
  \end{align*}
  which implies
  \[
  \|P_{\leq \mu} (v_1\overline{v_2}) \|_{L^2(\R^{1+3})}\ls
  C(\mu,\lambda_1)
  \|v_1\|_{U^4_{\mathbf{S}}}\|v_2\|_{U^4_{\mathbf{S}}}\ls
  C(\mu,\lambda_1)
  \|v_1\|_{V^2_{\mathbf{S}}}\|v_2\|_{V^2_{\mathbf{S}}},
  \]
  where we used $V^2_{\mathbf{S}}\hookrightarrow U^4_{\mathbf{S}}$,
  see \cite{HHetal-09}.
\end{proof}

\begin{pro}
For any $\mu,\lambda_1,\lambda_2\in 2^\Z$ and $u,v \in U^4_{\mathbf{S}}$,
\begin{equation}\label{eq:bil-str-hoeld}
\|P_\mu (P_{\lambda_1} u \overline{P_{\lambda_2} v})\|_{L^2(\R^{1+3})}
  \ls{} \min\{\mu,\lambda_1,\lambda_2\}^{\frac12} \lambda_1^{\frac{2-\alpha}{4}}\lambda_2^{\frac{2-\alpha}{4}} \|P_{\lambda_1} u\|_{U^4_{\mathbf{S}}}\|P_{\lambda_2} v\|_{U^4_{\mathbf{S}}}.
\end{equation}
\end{pro}

\begin{proof}
The Bernstein inequality implies
  \begin{equation*}
    \|P_\mu (P_{\lambda_1} u \overline{P_{\lambda_2} v})\|_{L^2(\R^{1+3})}\ls{} \mu^{\frac12}
                                                                           \|P_{\lambda_1} u \overline{P_{\lambda_2} v}\|_{L^2_t L^{\frac32}_x }
    \ls{}  \mu^{\frac12}
            \|P_{\lambda_1} u\|_{L^4_tL^3_x}\|P_{\lambda_2} v\|_{L^4_tL^3_x}
          \end{equation*}
and Lemma~\ref{lem:str} gives \eqref{eq:bil-str-hoeld} if $\mu\leq \lambda_1,\lambda_2$.
Similarly,
\begin{align*}
    \|P_\mu (P_{\lambda_1} u \overline{P_{\lambda_2} v})\|_{L^2(\R^{1+3})}
\ls{}& \|P_{\lambda_1} u\|_{L^4_tL^6_x}\|P_{\lambda_2} v\|_{L^4_tL^3_x}\\
\ls{}&  \lambda_1^{\frac12}
            \|P_{\lambda_1} u\|_{L^4_tL^3_x}\|P_{\lambda_2} v\|_{L^4_tL^3_x}.
\end{align*}
This concludes the proof because we can interchange the roles of $u$ and $v$.
\end{proof}
\begin{rmk}\label{rmk:second-proof}
In the case $1<\alpha \leq 2$ this gives another simple proof of a result which is slightly weaker than \eqref{eq:bil-rad-v2} but sufficient for our application:
  The obvious $U_{\mathbf{S}}^2$-version of \eqref{eq:bil-rad-v2} (see
  \cite[Prop.\ 2.19]{HHetal-09}) can now be interpolated with \eqref{eq:bil-str-hoeld} in the case $\mu\ls \lambda_1\sim\lambda_2$ via \cite[Prop.\ 2.20]{HHetal-09}, which gives
  \[
  \|P_\mu (P_{\lambda_1} u \overline{P_{\lambda_2} v})\|_{L^2(\R^{1+3})}\ls
  \mu^{\frac{3-\alpha}{2}}\Big(\frac{\mu}{\lambda_1}\Big)^{\frac{\alpha-1}{2}-\eps}\|P_{\lambda_1}
  u\|_{V^2_{\mathbf{S}}}\|P_{\lambda_2} v\|_{V^2_{\mathbf{S}}},
  \]
  for any fixed $\eps>0$.
\end{rmk}
From now on we set $X:=\{u \in X^0\mid u \text{ spatially radial}\}$.
\begin{cor}\label{cor:sum-bil-est}
 Let $1< \alpha\leq 2$.
  For all  $u,v\in X$, we have
  \begin{equation}\label{eq:sum-bil-est}
    \|(-\Delta)^{\frac{\alpha-3}{4}}(u \overline{v})\|_{L^2(\R^{1+3})}\ls \|u\|_{X}\|v\|_{X}.
  \end{equation}
\end{cor}
\begin{proof}
We decompose
\begin{align*}
\|(-\Delta)^{\frac{\alpha-3}{4}}(u \overline{v})\|_{L^2(\R^{1+3})}\ls{}& \sum_{\mu,\lambda_1,\lambda_2 \in 2^\Z}\mu^{\frac{\alpha-3}{2}}
\|P_\mu (u_{\lambda_1} \overline{v_{\lambda_2}})\|_{L^2(\R^{1+3})}\\
\ls{}& \Sigma_1+\Sigma_2+\Sigma_3,
\end{align*}
where $\Sigma_1$ is the contribution of $\lambda_1\ll \lambda_2\sim \mu$, $\Sigma_2$ is the contribution of $\lambda_2\ll \lambda_1\sim \mu$, and $\Sigma_3$ is the contribution of $\lambda_1\sim \lambda_2\gs\mu$.
From \eqref{eq:bil-str-hoeld} and Cauchy-Schwarz we obtain
\[
|\Sigma_1|\ls\sum_{\lambda_1\ll \lambda_2}\big(\frac{\lambda_1}{\lambda_2}\big)^{\frac{4-\alpha}{4}}\|P_{\lambda_1}
  u\|_{V^2_{\mathbf{S}}}\|P_{\lambda_2} v\|_{V^2_{\mathbf{S}}}\ls \|u\|_{X}\|v\|_{X}.
\]
Similarly we prove
\[
|\Sigma_2|\ls \|u\|_{X}\|v\|_{X}.
\]
Finally, using \eqref{eq:bil-rad-v2} we obtain
\[
|\Sigma_3|\ls \sum_{\mu\ls \lambda_1\sim\lambda_2}\Big(\frac{\mu}{\lambda_1}\Big)^{\frac{\alpha-1}{2}}\|P_{\lambda_1}
  u\|_{V^2_{\mathbf{S}}}\|P_{\lambda_2} v\|_{V^2_{\mathbf{S}}}\ls \|u\|_{X}\|v\|_{X},
\]
where we have exploited the radiality.
\end{proof}
\section{Proofs of the main results}\label{sec:proofs}
\subsection{Proof of Theorem \ref{thm:main}}\label{subsec:proof-main}
It suffices to consider positive times. We represent the solution  of \eqref{eqn} on $[0,\infty)$ using the Duhamel's formula
\begin{align}
  u(t)&=\mathbf{1}_{[0,\infty)}(t)S(t)\varphi+i\sigma \mathbf{J}(u,u,u)(t), \\
  \mathbf{J}(u_1,u_2,u_3)(t)
      &= \mathbf{1}_{[0,\infty)}(t)\int_0^t S(t-t')\big[\{ |\cdot|^{-\alpha}*(u_1 \overline{u_2}) \} u_3
        \big](t')dt'.
\end{align}

For all $\varphi\in H^s_{rad}$ we immediately have
\begin{equation}\label{eq:lin}
  \|\mathbf{1}_{[0,\infty)}\mathbf{S}\varphi\|_{X^s}
 \ls\|\varphi\|_{H^s}.
\end{equation}

Next, we study the nonlinear part. 
  For all $u_1,u_2,u_3\in X$, we have
  \begin{equation}\label{eq:nonlin-est}
    \big\| 
    \mathbf{J}(u_1,u_2,u_3)  \big\|_{X}
    \lesssim
    \|u_1\|_{X}
    \|u_2\|_{X}
    \|u_3\|_{X}.
  \end{equation} 
Indeed, we have
  \begin{align*}
    \|\mathbf{J}(u_1,u_2,u_3) \|_{X}
    &\ls \sup_{v\in X :\|v\|_X\leq 1}
      \Big|
      \iint (-\Delta)^{\frac{\alpha-3}{2}} (u_1\overline{u_2})u_3(t) \overline{v(t)}dtdx   
      \Big|
  \end{align*}
 by  a standard duality argument, see e.g.\ \cite{HHetal-09}.
  Further, we obtain
\begin{align*}
&\Big|
      \iint (-\Delta)^{\frac{\alpha-3}{2}} (u_1\overline{u_2})u_3(t) \overline{v(t)}dtdx   
      \Big|\\
\ls{}& \|(-\Delta)^{\frac{\alpha-3}{4}} (u_1\overline{u_2})\|_{L^2}\|(-\Delta)^{\frac{\alpha-3}{4}} (u_3\overline{v})\|_{L^2}\\
\ls{}& 
    \|u_1\|_{X}
    \|u_2\|_{X}
    \|u_3\|_{X}\|v\|_X
\end{align*}
by Corollary \ref{cor:sum-bil-est}, which implies \eqref{eq:nonlin-est}.

Theorem~\ref{thm:main} now follows from the standard approach via the contraction mapping
principle. In particular, the scattering claim follows from the fact
that functions in $V^2_{\mathbf{I}}$ have a limit at $\infty$. We omit
the details.

\subsection{Proof of Theorem
  \ref{thm:local}}\label{subsec:proof-local}
Let $\Lambda, r\geq 1$ be given. Recall that
\begin{equation*}
    B_{r,\Lambda} :=\big\{\varphi\in L_{rad}^2(\R^3) : \| \varphi\|_{L^2} \leq r, \|P_{>\Lambda} \varphi
    \|_{L^2} \leq \eta r^{-1} \big\},
\end{equation*}
where the parameter $0<\eta\leq 1$ will be determined below, and define
\begin{equation*}
    D_{R,\eps} :=\big\{u\in X : \| u\|_{X} \leq R, \|P_{>\Lambda} u
    \|_{X} \leq \eps \big\},
\end{equation*}
for some $0<\eps\leq R$. We implicitely assume that all functions are
supported in $[0,T]$. Split $\mathbf{J}(u)=\mathbf{J}_1(u_{\leq \Lambda},u_{>\Lambda})+\mathbf{J}_2(u_{\leq
  \Lambda},u_{>\Lambda})$, where $\mathbf{J}_1$ is at least quadratic in
$u_{>\Lambda}$ and $\mathbf{J}_2$ is at least quadratic in $u_{\leq \Lambda}$.
For $\mathbf{J}_1$, we use \eqref{eq:nonlin-est} and obtain
\[
\|\mathbf{J}_1(u_{\leq \Lambda},u_{>\Lambda})\|_X\ls R\eps^2,\]
for all $u \in D_{R,\eps}$.
We turn to  $\mathbf{J}_2$. First, we have
\begin{align*}
\|(-\Delta)^{\frac{\alpha-3}{2}}(u_{\leq \Lambda}\overline{u_{\leq \Lambda}})
  v\|_{L^1_tL^2_x}\ls \Lambda^\alpha \|u_{\leq \Lambda}\|_{L^2_tL^2_x}^2
  \|v\|_{L^\infty_t L^2_x}\ls T \Lambda^\alpha \|u\|^2_X\|v\|_X.
\end{align*}
Second, we have
\begin{align*}
&\|(-\Delta)^{\frac{\alpha-3}{2}}(u_{\leq \Lambda}\overline{v}_{>4 \Lambda})
  u_{\leq \Lambda}\|_{L^1_tL^2_x}\ls \Lambda^{\alpha-3}\|u_{\leq
                                   \Lambda} \overline{v}_{>4
                                   \Lambda}\|_{L^1_tL^2_x} \|u_{\leq \Lambda}\|_{L^\infty_t L^\infty_x}\\
&\ls T \Lambda^{\alpha-3} \|\overline{v}_{>4
                                   \Lambda}\|_{L^\infty_t L^2_x}
\|u_{\leq \Lambda}\|_{L^\infty_t L^\infty_x}^2 \ls T \Lambda^{\alpha}\|u\|^2_X\|v\|_X,
\end{align*}
and we obtain
\begin{align*}
\|(-\Delta)^{\frac{\alpha-3}{2}}(u_{\leq \Lambda}\overline{v})
  u_{\leq \Lambda}\|_{L^1_tL^2_x}\ls T \Lambda^{\alpha}\|u\|^2_X\|v\|_X,
\end{align*}
because the contribution of $v_{\leq 4 \Lambda}$ can be treated as in the
first estimate above.
We conclude that there exists a $C\geq 1$, such that for all $\varphi\in
B_{r,\Lambda}$ and all $u \in D_{R,\eps}$ we have
\begin{align}
\|\mathbf{1}_{[0,T)}\mathbf{S}\varphi\|_X \leq C r, \;
  \|\mathbf{1}_{[0,T)}\mathbf{S}\varphi_{>\Lambda}\|_X \leq C \eta r^{-1},\\
\|\mathbf{J}_1(u_{\leq \Lambda},u_{>\Lambda})\|_X\leq C R\eps^2, \;
  \|\mathbf{J}_2(u_{\leq \Lambda},u_{>\Lambda})\|_X\leq C T \Lambda^\alpha R^3,
\end{align}
and similar estimates for differences.
By choosing $R=4Cr$, $\eps=2^{-6}C^{-2}r^{-1}$, $T=2^{-18}C^{-6}r^{-4}\Lambda^{-\alpha}$,
and $\eta=2^{-10}C^{-3}$, one checks that
  \[
  D_{R,\eps}\rightarrow D_{R,\eps}, \qquad u \mapsto \mathbf{1}_{[0,T)}(\mathbf{S}\varphi+\mathbf{J}(u))
  \]
  is a strict contraction, for given $\varphi\in
  B_{r,\Lambda}$. The contraction mapping principle implies Theorem \ref{thm:local}. 

\subsection{Proof of Theorem~\ref{thm:illposed}}\label{subsec:proof-ill}
It suffices\footnote{In the published paper and previous
  versions of the preprint we incorrectly applied the abstract ill-posedness result established by Bejenaru and
  Tao in \cite[Proposition 1]{BT-06} to deduce the discontinuity of the flow map from the computation presented in this section. Thanks go to Robert Schippa for helpful discussions about the Bejenaru-Tao argument.} to show that if $s<0$, then
  \begin{equation}\label{ineq:tri2}
    \sup_{t\in[0,T]}
    \bigg\| \int_{0}^{t}S(t-t')
    (|\cdot|^{-\alpha}*|S(t')\varphi|^{2}
    S(t')\varphi)dt'
    \bigg\|_{H^{s}(\mathbb{R}^{3})}
    \lesssim
    \|\varphi\|_{H^{s}(\mathbb{R}^{3})}^{3}   
  \end{equation}
  fails to hold for some radial data
  $\varphi\in H^{s}(\mathbb{R}^{3})$. We modify the construction for the case $\alpha=1$ from \cite[Section 3]{HL-14} and define the annulus
  $A_{\lambda}=\{ \xi\in\mathbb{R}^{3}: \lambda\le|\xi|\le 2\lambda
  \}$.
  Let $\varphi$ be the inverse Fourier transform of the characteristic
  function $\mathbf{1}_{A_\lambda}$. Clearly, $\varphi$ is radial and
  $\|\varphi\|_{H^{s}(\mathbb{R}^{3})}\sim\lambda^{s+\frac 32}$. With this choice of $\varphi$, let
\[
 \Phi(t):=\int_{0}^{t}S(t-t')
    (|\cdot|^{-\alpha}*|S(t')\varphi|^{2}
    S(t')\varphi)dt'.
\]
We compute the spatial Fourier transform
  \begin{align*}
    &\widehat{\Phi}(t,\xi)=
      \int_{0}^{t}\int_{\mathbb{R}^{3}}
      S(t-t')
      \frac{\mathcal{F}_x(|S(t')\varphi|^{2})(\eta)}{|\eta|^{3-\alpha}}
      \mathcal{F}_xS(t')\varphi(\xi-\eta)d\eta
      dt' \\
    &=\int_{0}^{t}
      \iint\limits_{\mathbb{R}^{3}\times\mathbb{R}^{3}}
      e^{-i(t-t')|\xi|^{\alpha}}
      \frac
      {e^{-it'|\eta-\sigma|^{\alpha}}\mathbf{1}_{A_\lambda}(\eta-\sigma)
      e^{it'|\sigma|^{\alpha}}\mathbf{1}_{A_\lambda}(\sigma)}
      {|\eta|^{3-\alpha}}
      e^{-it'|\xi-\eta|^{\alpha}}\mathbf{1}_{A_\lambda}(\xi-\eta)
      d\sigma d\eta
      dt' \\
    &= e^{-it|\xi|^{\alpha}}
      \iint\limits_{\mathbb{R}^{3}\times\mathbb{R}^{3}}
      \int_{0}^{t}
      e^{it' g_{\alpha}(\xi,\eta,\sigma)} dt'
      \frac
      {\mathbf{1}_{A_\lambda}(\eta)
      \mathbf{1}_{A_\lambda}(\sigma)
      \mathbf{1}_{A_\lambda}(\xi-\eta-\sigma)}
      {|\eta+\sigma|^{3-\alpha}}
      d\sigma d\eta
  \end{align*}
where  
\[
    g_{\alpha}(\xi,\eta,\sigma)
      =|\xi|^{\alpha}-|\eta|^{\alpha}+|\sigma|^{\alpha}-|\xi-\eta-\sigma|^{\alpha}.
  \]
Choose $T= \epsilon\lambda^{-\alpha} $ with
  $0<\epsilon\ll1$. In the domain of integration we have
  \[
  |t g_\alpha(\xi,\eta,\sigma)| \lesssim |t\lambda^{\alpha}|\ll 1
  \]
 and we get
  \begin{equation}\label{lowbound:intt}
    \bigg|\int_{0}^{t}
    e^{it' g_{\alpha}(\xi,\eta,\sigma)} dt'\bigg|
    \gtrsim
    \int_{0}^{t}
    \cos{(t' g_{\alpha}(\xi,\eta,\sigma))} dt'
    \gtrsim
    |t| \ \text{for} \ t<T.
  \end{equation}
  Thus, if $\xi\in A_\lambda$, we have
  \[
  |\widehat{\Phi}(t,\xi)| \gtrsim
  |t|\int_{A_{\lambda}}\int_{A_{\lambda}}
  |\eta+\sigma|^{-3+\alpha}d\eta d\sigma \gtrsim
  \epsilon\lambda^{-\alpha}\lambda^{3+\alpha}.
  \]
  From this we easily obtain
  $\| \Phi \|_{H^s(\R^3)}\gtrsim \epsilon \lambda^{s+\frac{9}{2}}$.
  In conclusion, these norm calculations and \eqref{ineq:tri2} give
  \begin{align*}
    \epsilon\lambda^{s+\frac{9}{2}}
    \lesssim
    \|\Phi(t)\|_{H^{s}}
    \lesssim
    \|\varphi\|_{H^s}^3
    \sim
    \lambda^{3s+\frac{9}{2}}.
  \end{align*}
  So if $s<0$, \eqref{ineq:tri2} fails to hold as  $\lambda\to \infty$.

\bibliographystyle{amsplain} \bibliography{fnls-refs}

\end{document}